\documentclass[10pt]{amsart}
\DeclareUnicodeCharacter{FF0C}{,}
\usepackage{amssymb}
\usepackage{amsmath}

\usepackage{tikz}
\usepackage{url}
\usepackage{graphicx} 

\usepackage{amsmath}
\usepackage{amsfonts}
\usepackage{mathrsfs}
\usepackage{amsthm}
\usepackage{amssymb}
\usepackage{xcolor}
\usepackage[initials]{amsrefs}
\usepackage{calc}
\usepackage{enumitem}
\usepackage{tikz}
\usepackage{apptools}
\usepackage{url}
\usepackage{cancel}
\usepackage{apptools}
\usepackage{soul}

\usepackage{hyperref}
\hypersetup{
	pdfstartview={XYZ null null 1.00}, 
	pdfpagemode=UseNone, 
	colorlinks,
	breaklinks, 
	linkcolor=blue,
	urlcolor=blue, 
	citecolor=blue
}

\newtheorem{thm}{Theorem}[section]

\newtheorem{lem}[thm]{Lemma}
\newtheorem{defn}[thm]{Definition}

\newtheorem{cor}[thm]{Corollary}
\newtheorem{claim}[thm]{Claim}

\newcommand{\Parts}{\textrm{Parts}}
\newcommand{\len}{\textrm{len}}
\newcommand{\card}[1]{\left\lvert#1\right\rvert}
\newcommand{\Z}{\mathbb{Z}}
\newcommand{\R}{\mathbb{R}}
\newcommand{\A}{\mathcal{A}}
\newcommand{\Lang}{\textrm{Lang}}
\newcommand{\Per}{\textrm{Per}}
\newcommand{\Skel}{\textrm{Skel}}
\newcommand{\blank}{\square}
\newcommand{\isfilled}{\prec}
\newcommand{\Minv}[2]{\mathcal{M}_{#2}(#1)}
\DeclareEmphSequence{\bfseries}

\newlength{\squarewidth}
\newlength{\onewidth}
\newlength{\zerowidth}
\newcommand{\one}[0]{
    \hspace{(\squarewidth - \onewidth)/2}
    1\hspace{(\squarewidth - \onewidth)/2}
}
\newcommand{\zero}[0]{
    \hspace{(\squarewidth - \zerowidth)/2}
    0\hspace{(\squarewidth - \zerowidth)/2}
}

\usepackage{todonotes}
 
\usepackage{comment}

\renewcommand{\sigma}{S}

\begin{document}
\title{Generalized Oxtoby subshifts and hyperfiniteness}
\author{Konrad Deka}
\address{Faculty of Mathematics and Computer Science, Jagiellonian University in Krakow, ul. {\L}ojasiewicza 6, 30-348 Krak\'{o}w, Poland}
\email{deka.konrad@gmail.com}

\author{Bo Peng}
\address{Department of Mathmatics and Statistics, McGill University. 805 Sherbrooke Street West Montreal, Quebec, Canada, H3A 2K6}
\email{bo.peng3@mail.mcgill.ca}

\maketitle
\begin{abstract}
    We show that there exists a class of symbolic subshifts which realizes all Choquet simplices as simplices of invariant measures and the conjugacy relation on that class is hyperfinite.
\end{abstract}
\section{Introduction}
In this paper, we study the relationship between the complexity of conjugacy relation of symbolic subshifts and the complexity of homeomorphism of simplices of their invariant measures.

A theorem of Downarowicz \cite{DownarowiczChoquet} states that for every compact metric Choquet simplex $C$,
there exists a Toeplitz subshift
whose simplex of invariant measures is 
affinely homeomorphic to $C$. 
A theorem of Sabok \cite{MScompleteness} states that the relation of affine homeomorphism between 
compact metric Choquet simplices is a complete orbit equivalence relation. 
Naturally, if two subshifts $X, Y$ are isomorphic, 
then the simplices of invariant measures $\Minv{X}{\sigma}$, 
$\Minv{Y}{\sigma}$ are affinely homeomorphic.
One could hope that in some restricted class of subshifts, a converse statement holds,
which could lead to a nontrivial lower bound on the complexity of the conjugacy relation of Toeplitz subshifts.

 Gao, Jackson, Seward \cite[Theorem 9.3.3]{GaoJacksonSeward} and Thomas\cite{Simon-Thomaz} proved that the conjugacy relation for minimal symbolic subshifts is not smooth but it is not known whether it is hyperfinite or not. The exact complexity of the conjugacy relation on symbolic Toeplitz subshifts is a well-known open problem \cite[Question 1.3]{ST}, however, the complexity of conjugacy relation restricted on some subclasses has been computed. Kaya \cite{Kaya_2017} showed that the conjugacy of Toeplitz systems with growing blocks is hyperfinite which generalizes a result of Sabok and Tsankov \cite{ST} on Toeplitz systems with separate holes. Those two classes consist of regular Topelitz systems which implies all such systems are uniquely ergodic.  

It turns out the subshifts constructed by Downarowicz  belong to the class of generalized Oxtoby subshifts. 
The main result of this paper is that the conjugacy relation of that class is hyperfinite.

\begin{thm}\label{Main}
    The conjugacy relation of generalized Oxtoby systems is hyperfinite.
\end{thm}

As a consequence, there exists a class of Toeplitz subshifts, such that on one hand every Choquet simplex can be realized as the set of invariant measures on one of those subshifts, and on the other hand, the conjugacy relation on that class is hyperfinite.

We give two proofs of Theorem \ref{Main}. The structure of the paper is the following. Section \ref{2} contains preliminaries. Section \ref{3} contains basic definition in symbolic subshifts. In Section \ref{4}, we review the result of Downarowicz on realizing Choquet simplices. Section \ref{5} contains analysis of symbolic and Toeplitz subshifts.  Section \ref{6} contains the first proof of Theorem \ref{Main}.  Section \ref{7} contains analysis of the Oxtoby property for subshifts and Section \ref{8} contains the second proof of Theorem \ref{Main}.

\section{Preliminaries} \label{2}
A \emph{topological dynamical system}, tds for short, is a pair $(X, T)$ where $X$ is a compact metric space and $T \colon X \to X$ is a continuous map.
For $x \in X$, its \emph{orbit} is $\mathcal{O}(x) := \{T^n x \colon n \ge 0\}$. 
A tds $(X, T)$ is \emph{minimal} if it has no nontrivial subsystem. 
Equivalently, it is minimal if and only if $\overline{\mathcal{O}(x)} = X$ for all $x \in X$.

Let $(X, T)$ be a tds, and let $\mu$ be a Borel probability measure on $X$.
We say that $\mu$ is \emph{$T$-invariant} if $\mu(A) = \mu(T^{-1} A)$ 
for all Borel $A \subseteq X$. 
We will write $\Minv{X}{T}$ for the set of all such measures on $X$. 
This set is guaranteed to be nonempty.
Let $C(X)$ be the Banach space of all continuous functions $X \to \R$,
and consider the dual $C(X)^*$ with the weak* topology. 
Then $\Minv{X}{T}$ can be viewed as a compact convex subset of $C(X)^*$. 

A \textbf{Choquet simplex} $X$ is a compact convex, metrizable space space such that all $x\in X$ can be uniquely represented by a measure supported on its extreme points.

Let $\A$ be a finite set. Then $\A^* := \bigcup_{n \ge 0} \A^n$ is the family of all 
finite sequences with entries from $\A$, including the empty sequence of length zero.
If $\A$ is finite, then $\A^\Z$ is compact, 
and we use $\sigma \colon \A^\Z \to \A^\Z$ to denote the shift map. 
A \emph{subshift} is a closed set $X \subseteq \A^\Z$ such that $\sigma X = X$. 
Sometimes, we also call the topological dynamical system $(X, \sigma|_X)$ a subshift.
We also define the \emph{language of a subshift $X$} to be 
$$
    \Lang(X) := \left\{ x[i, j) \colon x \in X \textrm{ and } i \le j \right\}.
$$

A \textbf{Polish space} is a separable completely metrizable topological space. An \textbf{equivalence relation} $E$ on a Polish space is \textbf{Borel} if it is a Borel subset of $X^2$ and is \textbf{countable} if every equivalence class is countable. A countable Borel equivalence relation $E$ is \textbf{hyperfinite} if $E$ can be written as an increasing union of finite Borel equivalence relations.

Let $X$ and $Y$ be two Polish spaces and $E,F$ be two equivalence relations defined on $X$ and $Y$, respectively. A \textbf{Borel reduction} is a Borel map from $X$ to $Y$ such that
$$
x_1Ex_2\,\,\Leftrightarrow\,\,f(x_1)Ff(x_2). 
$$
An equivalence relation on a Polish space is \textbf{smooth} if it is Borel reducible to $=$ on $\mathbb{R}$.

\section{Toeplitz subshifts and generalized Oxtoby subshifts} \label{3}
Toeplitz subshifts and Toeplitz sequences are a standard method of constructing 
minimal subshifts, see for example \cite{DownarowiczSurvey} for the survey of the topic.
We will follow closely the notation from \cite{Kaya_2017}. 
An element $x \in \A^\Z$ is called a \emph{Toeplitz sequence} 
if for all $i \in \Z$ there exists $n \ge 1$ such that 
$x(i) = x(i + kn)$ for all $k \in \Z$. 
A \emph{Toeplitz subshift} is a subshift that is of the form 
$\overline{\mathcal{O}(x)}$ for some Toeplitz sequence $x$.
Every Toeplitz subshift is minimal.
We introduce the following notions:
\begin{itemize}
    \item For $x \in \A^\Z$ and an integer $p > 1$, define the \emph{$p$-periodic part of $x$},
    $$
        \Per_p(x) := \left\{
            i \in \Z \colon x(i)=x(i + kp) \textrm{ for all } k \in \Z
        \right\}.
    $$
    Clearly, if $r$ divides $p$, then $\Per_r(x) \subseteq \Per_p(x)$.
    We say that \emph{$p$ is an essential period of $x$} if
    $\Per_p(x) \not \subseteq \Per_r(x)$ for all $r < p$. 

    \item A number $i \in \Z$ is called a \emph{$p$-hole} if $i \not \in \Per_p(x)$.

    \item The \emph{$p$-skeleton of $x$} is the sequence $\Skel(x, p)$
    defined as 
    \begin{equation*}
        \Skel(x, p)(i) := 
        \begin{cases}
            x(i), & \textrm{ if } i \in \Per_p(x), \\
            \blank & \textrm{ otherwise.}
        \end{cases}
    \end{equation*}
    We will call $\blank$ the blank symbol. 
    Observe that $\Skel(x, p)$ is a $p$-periodic sequence.
    
    \item Now let $x$ be a Toeplitz sequence.
    A \emph{period structure for $x$} is a strictly increasing
    sequence of integers $(p_t)_{t \ge 0}$ such that 
    $p_0 = 1$, $p_t$ divides $p_{t + 1}$ for all $i$, and 
    $$
        \bigcup_{t \ge 1} \Per_{p_t}(x) = \Z.   
    $$
    As a consequence of definitions, for every Toeplitz sequence $x$ 
    there exists a period structure for it.
\end{itemize}

Given some sequences $x, y \in (\A \sqcup \{ \blank \})^\Z$, we will write 
$x \isfilled y$ if $x$ can be obtained by replacing some of blank symbols in $y$
by symbols in $\A$. In other words, $x \isfilled y$ iff 
$y[i] \neq \blank$ implies $x[i] = y[i]$ for all $i \in \Z$.

We will now describe a general method of producing 
a Toeplitz sequence $x$ over the alphabet $\A$.
First, choose any sequence $1 = p_0 < p_1 < p_2 < \dots$ such that $p_t$ divides $p_{t+1}$,
which will be a period structure for $x$. 
Roughly speaking, one can fill some positions in $x$ by 
symbols that repeat with period $p_1$, then fill out some of the remaining gaps
by symbols that repeat with period $p_2$, and so on.
More formally, for all $t \ge 0$, 
we will define $x_t \in (\A \sqcup \{ \blank \})^\Z$, such that
$x_t$ is $p_t$-periodic and $x_{t+1} \isfilled x_t$.
We will use blank symbol $\blank$ to represent 
a position waiting to be filled at a later step.
\begin{enumerate}[label=(\roman*)]
    \item Begin with $p_0 = 1$ and $x_0 := \dots \blank \blank \blank \dots$.
        
    \item We define $x_{t+1}$ for $t \ge 0$ by induction.
        We have $x_t$ already defined, and it is $p_t$-periodic, 
        therefore $p_{t + 1}$-periodic. 
        Let 
        $R_{t + 1} := \left\{i \in [0, p_{t+1}) \colon x_t(i) = \blank \right\}$.
        Pick a set 
        $I_{t + 1} \subseteq 
        R_{t + 1}$
        and an assignment of symbols $\phi_{t + 1} \colon I_{t + 1} \to \A$.
        Define
        \begin{equation*}
            x_{t + 1}(i) := 
            \begin{cases}
                \phi_{t + 1}(r), & \textrm{ if } i \equiv r \pmod{p_{t + 1}}
                \textrm{ and } r \in I_{t + 1},\\
                x_t(i), & \textrm{ otherwise. } 
            \end{cases}
        \end{equation*}
        It is clear that $x_{t + 1}$ is $p_{t + 1}$-periodic and 
        $x_{t + 1} \isfilled x_t$.

    \item If we have filled out all the positions
    (i.e. for all $i \in \Z$, there exists $t$ such that $x_t[i] \neq \blank$),
    then $x_t$ converges to $x \in \A^\Z$ which is a Toeplitz sequence. 
    For example, one can 
    choose sets $I_t$ in such a manner that $x_t[-t, t)$ contains no blank.
        
\end{enumerate}
Below we illustrate possible first two steps of the construction, 
where $p_1 = 4, p_2 = 8$. 
\begin{align*}
	x_0 := & 
    \dots
    \square \square \square \square \square \square \square \square
    \square \square \square \square \square \square \square \square
    \square \square \square \square \square \square \square \square
    \square \square \square \square \square \square \square \square
    \dots, \\
	x_1 := & 
    \dots
    \square \one \square \square \square \one \square \square
    \square \one \square \square \square \one \square \square
    \square \one \square \square \square \one \square \square
    \square \one \square \square \square \one \square \square
    \dots, \\
	x_2 := & 
    \dots
    \square \one \zero \zero \square \one \square \square
    \square \one \zero \zero \square \one \square \square
    \square \one \zero \zero \square \one \square \square
    \square \one \zero \zero \square \one \square \square
    \dots
\end{align*}
It follows from the construction that $\Skel(x, p_t) \isfilled x_t$. 
We warn the reader that we might have $\Skel(x, p_t) \neq x_t$ 
(continuing the example, suppose we would fill all the remaining blanks with 
$1$ in the next step of the construction. 
Then we would obtain an $x$ that is $8$-periodic, 
and then $\Skel(x, p_2) = x \neq x_2$). 

Oxtoby \cite{Oxtoby1952} has given an example of a specific Toeplitz sequence 
whose orbit closure is a minimal, 
but not uniquely ergodic topological dynamical system.
Roughly speaking, his construction follows the method outlined above, 
where one takes a sequence $(p_t)_{t \ge 1}$ growing sufficiently fast, 
then takes $I_{t + 1} = ([0, p_t) \cup [p_{t+1} - p_t, p_{t+1})) \cap R_{t + 1}$
and $\phi_{t + 1}$ that is everywhere $0$ or $1$ depending on parity of $t$.
Extending his methods, Williams \cite{Williams} 
and Downarowicz \cite{DownarowiczChoquet}
have further studied the possible behavior of the set of invariant measures 
in Toeplitz systems.
This led to the following definition introduced in \cite{DKL}:
\begin{defn}
    Let $x$ be a Toeplitz sequence, 
    and $(p_t)_{t \ge 1}$ a period structure for $x$. 
    We say that $x$ is a 
    \emph{generalized Oxtoby sequence with respect to $(p_t)_{t \ge 1}$}
    if for each $t \ge 1$ and $k \in \Z$, 
    if $x[k p_t, (k + 1) p_t) $ contains a $p_{t+1}$-hole,
    then all $p_t$-holes in this interval are $p_{t+1}$-holes. 
\end{defn}
\begin{defn}
    A subshift $X \subseteq \A^\Z$ is a 
    \emph{generalized Oxtoby subshift with respect to $(p_t)_{t \ge 1}$}
    if $X = \overline{\mathcal{O}(x)}$ for some $x$ 
    that is a generalized Oxtoby sequence with respect to $(p_t)_{t \ge 1}$.
\end{defn}

We will see that in the class of generalized Oxtoby subshifts 
(with respect to some fixed period structure) 
the isomorphism relation is hyperfinite. 
Simultaneously, every Choquet simplex can arise as the simplex of invariant measures
of such subshift.

\section{Simplices of invariant measures for generalized Oxtoby subshifts} \label{4}
In \cite{DownarowiczChoquet}, Downarowicz has proven that every compact metric 
Choquet simplex can be realised as the simplex of invariant measures of some Toeplitz
subshift $X$ over the alphabet $\{0, 1\}$. We will examine this proof to see that 
$X$ can be constructed so that it is a generalized Oxtoby subshift. We begin by reviewing some definitions from \cite{DownarowiczChoquet}:

\begin{defn}
    Let $b \in \{0, 1\}^n$ for some $n \ge 1$. For $b_0 \in \{0, 1\}^*$,
    we define
    \begin{equation*}
        F^*_{b_0}(b) := 
        \frac{1}{\len(b_0)}
        \card{ \left\{
            m \in \mathbb{N} \colon 0 \le m \le \len(b_0) - \len(b) \wedge
            b_0[m, m + len(b)) = b
        \right\} },
    \end{equation*}
    \begin{multline*}
        F^{**}_{b_0}(b, k, j) := 
        \frac{1}{\len(b_0)}
        \bigl| \bigl\{
            m \in \mathbb{N} \colon 0 \le m \le \len(b_0) - \len(b) 
             \textrm{ and } \\ b_0[m, m + len(b)) = b
            \textrm{ and } m \equiv k \textrm{ mod } j
        \bigr\} \bigr|.
    \end{multline*}
    Let $v^*(b) := 1 / 2^{\len(b)}$. 
    Fix also a sequence of positive real numbers $(c_k)_{k \ge 1}$, with 
    $\sum_{k \textrm{ odd}} c_k = 1$.
    For a triple $(b, k, j)$ with $k$ odd and $0 \le j < k$, let
    $v^{**}(b, k, j) := c_k / 2^{\len(b)}$.
    Finally, for a measure $\mu$ on $\{0, 1\}^\mathbb{Z}$, define
    \begin{align*}
        d^*(b_0, \mu) := & \sum_{b \in \{0, 1\}^*}
            \lvert F^*_{b_0}(b) - \mu([b]) \rvert v^*(b), \\
        d^{**}(b_0, \mu) := & \sum_{
        \substack{
        b \in \{0, 1\}^* \\
        k \textrm{ odd}, 0 \le j < k}
        }
        \left\lvert F^{**}_{b_0}(b, k, j) - \frac{\mu([b])}{k} \right\rvert v^{**}(b).
    \end{align*}
\end{defn}

\begin{defn}
    Let $Y \subseteq \{0, 1\}^\mathbb{Z}$ be a subshift.
    We use $\mathcal{M}_0(Y)$ to denote the set of those measures $\mu \in \Minv{Y}{\sigma}$ 
    which satisfy 
    $$d^*(b_n, \mu) \to 0 \; \Rightarrow \; d^{**}(b_n, \mu) \to 0 $$
    for all sequences $(b_n)_{n \ge 0}$ of words from $\Lang(Y)$. 
\end{defn}

From \cite{DownarowiczChoquet} one can deduce the following:
\begin{thm}
    For every compact metric Choquet simplex $C$ there exists 
    a Toeplitz subshift $X$ over the alphabet $\{0, 1\}$
    such that $\Minv{X}{\sigma}$ is affinely homeomorphic to $C$.
    Moreover, one can find such $X$ that is additionally a 
    generalized Oxtoby subshift with respect to $(p_t)_{t \ge 1}$, where 
    $p_t = \prod_{i=1}^t 2^{i + 1}$.
\end{thm}

\begin{proof}
    The first part of theorem is stated explicitly in 
    \cite{DownarowiczChoquet}*{Theorem 5}.
    The "moreover" part follows from a closer reading of its proof. 
    Let us restate the relevant steps:
    \begin{enumerate}[label=(\roman*)]
        \item Given the simplex $C$, there exists a (not necessarily minimal)
        subshift $Y \subseteq \{0, 1\}^\Z$ such that $\Minv{Y}{\sigma}$ 
        is affinely homeomorphic to $C$, and $\mathcal{M}_0(Y) = \Minv{Y}{\sigma}$
        \cite{DownarowiczChoquet}*{Theorem 3}.
        
        \item \label{label:31.ii} 
        One selects a sequence $(b_t)_{t \ge 1}$ such that $b_t \in \Lang(Y)$,
        $|b_t| = \prod_{i = 1}^t (2^i - 1)$, 
        and for every $w \in \Lang(Y)$ there exist infinitely many $t$ such that 
        $w$ is a prefix of $b_t$. 

        \item \label{label:31.iii} We carry out a construction similar to the one in previous section.
        Define $x_1$ by setting $x_1(i) = b_1$ if $i \equiv 0 \pmod{4}$, otherwise $x_1(i) = \blank$.
        After step $t$, we will have $x_t$ that is $p_t$-periodic,
        with $\prod_{i = 1}^t (2^i - 1) = |b_t|$ blank symbols in
        the interval $[0, p_t)$. 
        If $t$ is odd, we form $x_{t + 1}$ by replacing the blanks in
        each interval $[0, p_t) + kp_{t + 1}$ 
        with the consecutive symbols from $b_{t + 1}$. 
        If $t$ is even, we form $x_{t + 1}$ by replacing the blanks in
        each interval $[-p_t, 0) + kp_{t + 1}$ 
        with the consecutive symbols from $b_{t + 1}$. 
        As $t \to \infty$, we have $x_t \to x \in \{0, 1\}^\Z$ 
        that is a Toeplitz sequence. 
        Using the equality $\mathcal{M}_0(Y) = \Minv{Y}{\sigma}$, Downarowicz shows that
        $X := \overline{\mathcal{O}(x)}$ is a Toeplitz subshift such that
        $\Minv{X}{\sigma}$ is affinely homeomorphic to $\Minv{Y}{\sigma}$. 
        By (i) we have that $\Minv{Y}{\sigma}$ is affinely homeomorphic to $C$,
        which ends the proof of the first part.
    \end{enumerate}

    We claim that $x$ is a generalized Oxtoby sequence w.r.t. $(p_t)_{t \ge 1}$.
    If $Y = \{0^\infty\}$ or $Y = \{1^\infty\}$ then 
    $x = 0^\infty$ or $x = 1^\infty$ and the claim holds.
    
    For $Y \neq \{0^\infty\}, \{1^\infty\}$, 
    we claim that $\Skel(x, p_t) = x_t$. 
    Recall that $\Skel(x, p_t) \isfilled x_t$.
    So it suffices to show that if $x_t(i) = \blank$,
    then $\Skel(x, p_t)(i) = \blank$ too.
    Both $\Skel(x, p_t)$ and $x_t$ are $p_t$-periodic, 
    so without loss of generality, we assume that $i \in [0, p_t)$. 
    Suppose that $x_t(i)$ is the $k$-th blank in $x_t[0, p_t)$.
    Let $\alpha \in \{0, 1\}$. 
    There exists $w \in \Lang(Y)$ such that $w(k) = \alpha$. 
    By \ref{label:31.ii}, 
    there exists $T \ge t$ such that $w$ is a prefix of $b_{T}$,
    so $b_{T}(k) = \alpha$.
    During the steps $t + 1, \dots T$
    of the construction described in \ref{label:31.iii},
    we always fill out entire intervals of form $[kp_t, (k+1)p_t)$.
    For this reason, during the $T$-th step, we will put the symbol $\alpha = b_T(k)$
    at $x_T(i_\alpha)$ for some $i_\alpha \equiv i \pmod{p_t}$.
    So we have $x(i_1) = 1 \neq 0 = x(i_0)$ 
    and $i_1 \equiv i \equiv i_0 \pmod{p_t}$, 
    we conclude that $\Skel(x, p_t)(i) = \blank$.
    
    Finally, look again at the steps described in \ref{label:31.iii}.
    For all $t$, in $t+1$-th step we fill out entire intervals of
    form $[kp_t, (k+1)p_t)$.
    Since $\Skel(x, p_t) = x_t$ for all $t$,
    we conclude that $x$ is a generalized Oxtoby sequence w.r.t. 
    $(p_t)_{t \ge 1}$. 
\end{proof}

\section{Further analysis of Toeplitz subshifts} \label{5}
\begin{defn}
    Let $X$ be a Toeplitz subshift, and $x \in X$.
    Let $p > 1$ be any integer. For $0 \le k < p$, we define
    $$
        A(x, p, k) := \{ \sigma^i x \colon i \equiv k \pmod{p} \}.
    $$

     Let 
        $$
            \Parts(X, p) := \{ \overline{A(x, p, k)} \colon 0 \le k < p \}.
        $$
\end{defn}
The next lemma will show that the definition of $\Parts(X, p)$ does not depend on the choice of $x$. It is similar to Lemmas 6, 7 from \cite{Kaya_2017}, 
who in turn cites \cite{Williams}. 

\begin{lem} \label{lem:parts_of_toeplitz}
    Let $X$ be a Toeplitz subshift, and $x \in X$.
    Let $p > 1$ be any integer. For $0 \le k < p$,
    \begin{enumerate}[label=(\roman*)]
        \item We have
        $\sigma \overline{A(x, p, k)} = \overline{A(x, p, k + 1)}$.
        \item For all $x \in \overline{A(x, p, k)}$, we have 
        $\Skel(x, p) = \Skel(\sigma^k x, p)$. 
        In other words, all elements of $\overline{A(x, p, k)}$
        have the same $p$-skeleton.
        \item For any $0 \le k, l < p$, the sets 
        $\overline{A(x, p, k)}$ and $\overline{A(x, p, l)}$
        are either disjoint or equal. 
        \item  $\Parts(X, p)$ is a clopen partition of $X$,
        independent of the choice of $x$.
    \end{enumerate}
\end{lem}

\begin{proof}
    Item (i) is straightforward.
    Item (ii) is precisely the Lemma 6 of \cite{Kaya_2017}.
    Next, we claim that if $y \in \overline{A(x, p, k)}$,
    then $\overline{A(x, p, k)} = \overline{A(y, p, 0)}$.
    For a suitable sequence $n_j$ we have $\sigma^{n_j} x \to y$,
    and $n_j \equiv k$ for all $j$. 
    Take $z \in \overline{A(y, p, 0)}$, 
    so for a suitable sequence $m_j$ we have $\sigma^{m_j} y \to z$,
    and $m_j \equiv 0$ for all $j$ (here we use $\equiv$ to mean equality modulo $p$). 
    Consider any $\epsilon > 0$ and a compatible metric $d$ on $X$.
    Take $l$ such that $d(\sigma^{m_l} y, z) < \epsilon / 2$. 
    Since the map $\sigma^{m_l}$ is continuous and $\sigma^{n_j} x \to x$,
    we will have $d(\sigma^{m_l + n_j} x, z) < \epsilon$ for large enough $j$.
    This works for every $\epsilon > 0$ and $m_l + n_j \equiv k$, 
    so we conclude that $z \in \overline{A(x, p, k)}$.
    We've shown $\overline{A(x, p, k)} \supseteq \overline{A(y, p, 0)}$.
    Since the system is minimal, we have $\sigma^{n_j'}y \to x$ for some
    sequence of integers $n_j'$, and by choosing a subsequence, we can assume
    that $n_j' \equiv \alpha$ for some $\alpha$, for all $j$.
    If $z \in \overline{A(x, p, k)}$, then 
    then $\sigma^{m_j'} x \to z$ and $m_j' \equiv k$ for all $j$,
    and by the same argument as before we conclude that 
    $z \in \overline{A(y, p, k + \alpha)}$. Altogether, we got
    \begin{equation} \label{eqn:1}
        \overline{A(y, p, 0)} \subseteq
        \overline{A(x, p, k)} \subseteq
        \overline{A(y, p, k + \alpha)}.
    \end{equation}
    If $k + \alpha \equiv 0$, then 
    $\overline{A(y, p, 0)} = \overline{A(y, p, k + \alpha)}$ and the claim is proven.
    Otherwise, by repeated application of (i), we get
    $$
        \overline{A(y, p, 0)} \subseteq
        \overline{A(y, p, k + \alpha)} \subseteq
        \overline{A(y, p, 2(k + \alpha))} \subseteq
        \dots \subseteq
        \overline{A(y, p, p(k + \alpha))} = \overline{A(y, p, 0)},
    $$
    so again 
    $\overline{A(y, p, 0)} = \overline{A(y, p, k + \alpha)}$ and the claim is proven.

    With this claim, the remaining items are readily proven. 
    For (iii), suppose that the sets are not disjoint, so
    $y \in \overline{A(x, p, k)}$ and $\overline{A(x, p, l)}$. 
    Then we have 
    $ \overline{A(x, p, k)} = 
    \overline{A(y, p, 0)} = 
    \overline{A(x, p, l)}$. 
    
    For (iv), note that $\Parts(X, p)$ is a collection of disjoint sets by (iii).
    Next, we claim that $\bigsqcup_{k=0}^{p-1} \overline{A(x, p, k)} = X$. 
    Indeed, for any $y \in X$ we can find a sequence $n_j$ such that 
    $\sigma^{n_j} x \to y$, and after moving to a subsequence we can assume
    that $n_j \equiv k$ for some $k$ and all $j$, 
    so $y \in \overline{A(x, p, k)}$. 
    So $\Parts(X, p)$ is a finite closed partition, which implies that each set 
    in the partition is clopen. 
    It remains to show that it doesn't depend on the choice of $x$.
    Take any $y \in X$. 
    For some $k$, we have $y \in \overline{A(x, p, k)}$.
    By the claim, we have $\overline{A(y, p, 0)} = \overline{A(x, p, k)}$.
    By repeated application of (i), the sets 
    $\overline{A(y, p, i)}, i = 0 \dots p - 1$ are the same as 
    $\overline{A(x, p, i)}, i = 0 \dots p-1$ up to cyclic permutation.
\end{proof}
\begin{lem}\label{essential period}
   Let $X$ be a Toeplitz subshift and $x\in X$. Suppose $p$ is an essential period of $x$, then we have $\overline{A(x,p,m)}=\overline{A(x,p,m')}$ if and only if $m\equiv m'$ (mod $p$).
\end{lem}
\begin{proof}
    It is clear that if $m\equiv m'$ (mod $p$) then $\overline{A(x,p,m)}=\overline{A(x,p,m')}$. For the other direction, since $p$ is an essential period of $x$, for all $0\leq m,m'< p$ with $m\neq m'$, we have ${\rm Skel}(\sigma^mx,p) \neq {\rm Skel}(\sigma^{m'}x,p)$. Since all elements in $\overline{A(x,p,m)}$ have the same $p$-skeleton as $S^mx$, we obtain that $\overline{A(x,p,m)}$ and $\overline{A(x,p,m')}$ are disjoint.
\end{proof}
\textbf{Remark.}
Suppose that $p$ is an essential period of $x$, 
and $X=\overline{\mathcal{O}(x)}$.
Then $\Skel(\sigma^j x, p)$, for $j=0 \dots p-1$, are pairwise different.
By item (ii) of the lemma above, this implies that
the sets $\overline{A(x, p, j)}$, for $j=0 \dots p-1$ are pairwise different.
We warn, however, that $\Parts(X, p)$ may contain less than $p$ parts.

\vspace{\baselineskip}

We will repeat some notation introduced in \cite{Kaya_2017}.
Again, let $X$ be a Toeplitz subshift, and $p > 1$ an integer.
Let $W \in \Parts(X, p)$. 
By Lemma \ref{lem:parts_of_toeplitz}(ii), 
all elements of $W$ have the same $p$-skeleton, which we denote by $\Skel(W, p)$.
By the same lemma, 
$X$ is partitioned into sets of form $\overline{A(x, p, k)}$, and each element 
of $\overline{A(x, p, k)}$ has the same skeleton as $\sigma^k x$.
So all elements of $X$ have the same $p$-skeleton up to shifting. 
A \emph{filled $p$-block} of a $p$-skeleton $s$ (where $s = \Skel(x, p)$ or $s = \Skel(W, p)$ for some $x$ or $W$) is a run of consecutive non-blank elements, 
preceded and followed by a blank.
We also define
$$
    \Parts_*(X, p) := 
    \left\{
        W \in \Parts(X, p) \colon
        \Skel(W, p)(-1) = \blank 
        \textrm{ and }
        \Skel(W, p)(0) \neq \blank
    \right\}.
$$
Note that elements of $\Parts_*(X, p)$ 
are in 1-to-1 correspondence with 
the filled $p$-blocks of the skeleton $s$.
Given $W \in \Parts_*(X, p)$,
we define $len(W) := \min \{ i \ge 0 \colon \Skel(W, p)(i) = \blank \}$.
In other words, $len(W)$ is the length of the filled $p$-block 
starting in position $0$ of $\Skel(W, p)$.

We move to analysing isomorphisms of Toeplitz subshifts. 
\begin{defn}
    Let $X \subseteq \A^\Z$.
    A function $F \colon X \to \A^\Z$ is called a \emph{block code}
    if it is of the form
    $$
        F(x)(n) = f(x_{n - k}, x_{n - k + 1}, \dots x_{n + k- 1}, x_{n + k})
        \textrm{ for all } x \in X, n \in \Z,
    $$
    for some integer $k \ge 0$ and function $f \colon \A^{2k + 1} \to \A$ is an arbitrary function.
    When $F$ is a block code, we refer to the smallest such $k$ as the \emph{width}
    or \emph{radius} of the block code $F$.
\end{defn}

The following theorem is well-known: 
\begin{thm}[Curtis-Hedlund-Lyndon] \label{thm:CHL}
    Let $X, Y \subseteq \A^\Z$ be two subshifts.
    Suppose $F \colon X \to Y$ is an isomorphism from $(X, \sigma)$ to 
    $(Y, \sigma)$.
    Then $F$ is a block code.
\end{thm}

Let $p \ge 1$ be an integer,
and let $\phi \colon \A^p \to \A^p$ be a bijection.
We will write $Sym(\A^p)$ for the set of all such bijections.
We introduce a map $\tilde{\phi} \colon \A^\Z \to \A^\Z$, 
defined by 
$$
    \tilde{\phi}(x)[kp, (k+1)p) := \phi(x[kp, (k+1)p)) 
    \textrm{ for all } x \in \A^\Z, k \in \Z.
$$

Applying Theorem \ref{thm:CHL} to Toeplitz subshifts, 
Downarowicz, Kwiatkowski and Lacroix \cite{DKL} have proved:
\begin{thm} \cite{DKL}*{Theorem 1} \label{thm:DKL}
    Let $x, y$ be Toeplitz sequences with period structure $(p_t)$. 
    The following conditions are equivalent:
    \begin{enumerate}[label=(\roman*)]
        \item there exists an isomorphism of subshifts $f \colon 
        \overline{\mathcal{O}(x)} \to 
        \overline{\mathcal{O}(y)}
        $
        that sends $x$ to $y$,
        \item for some $p_t \ge 1$, 
        there exists a function $\phi \in Sym(\A^{p_t})$
        such that
        $\tilde{\phi} (x) = y$.
    \end{enumerate}
\end{thm}

Let $K(\A^\Z)$ denote the family of compact subsets of $\A^\Z$ 
(with Vietoris topology).
On $K(\A^\Z)$ we introduce relation $D_p$, defined by 
$$
    A D_p B \Leftrightarrow 
    B = \tilde{\phi}(A) \textrm{ for some } \phi \in Sym(\A^p).
$$

\begin{lem} \label{lem:parts_Dp}
    Let $p \ge 1$ be an integer and $X, Y$ be two Toeplitz subshifts. 
    If $A \in \Parts(X, p)$, $B \in \Parts(Y, p)$ 
    are such that $A D_p B$, then $X$ and $Y$ are isomorphic.
\end{lem}

\begin{proof}
    By Lemma \ref{lem:parts_of_toeplitz}, $A$ is clopen (and nonempty) in $X$.
    We can find a Toeplitz sequence $x \in A$
    (since $X = \overline{\mathcal{O}(x_0)}$ for some 
    Toeplitz sequence $x_0$,
    we have $\sigma^k x_0 \in A$ for some $k$, 
    so we can take $x := \sigma^k x_0$).
    Let $\phi \in Sym(\A^p)$ be such that $\tilde{\phi}(A) = B$. 
    Then $\phi(x) \in B \subseteq Y$ is also a Toeplitz sequence.
    Applying Theorem \ref{thm:DKL} we get that $X, Y$ are isomorphic.
    A different proof  
    is given in \cite{Kaya_2017}*{Lemma 10}.
\end{proof}

For finite sets $\mathscr{A}, \mathscr{B} \subset K(\A^\Z)$, Kaya introduces
one more relation
$$
    \mathscr{A} D_p^{fin} \mathscr{B}
    \Leftrightarrow
    \left\{ [A]_{D_p} \colon A \in \mathscr{A} \right\}
    =
    \left\{ [B]_{D_p} \colon B \in \mathscr{B} \right\}.
$$

\section{First proof of hyperfiniteness of the conjugacy of generalized Oxtoby subshifts} \label{6}
Fix a strictly increasing sequence $(p_t)_{t \ge 1}$ 
such that $p_t$ divides $p_{t + 1}$ for all $t$. 

\begin{lem} \label{lem:small_and_large_gap}
    Let $X$ be a generalized Oxtoby subshift 
    with respect to $(p_t)_{t \ge 1}$.
    Pick any $c \ge 1$.
    For large enough $t$,
    every $W \in \Parts_*(X, p_t)$ has either
    $len(W) \ge p_{t-1} + 2c$, or $len(W) \le p_{t-1}$.
\end{lem}

\begin{proof}
    Pick any $x \in X$ that is 
    a generalized Oxtoby sequence w.r.t. $(p_t)_{t \ge 1}$.
    Pick $t$ large enough so that 
    $x[i + kp_{t-1}] = x[i]$ for all $k \in \Z$ and $|i| \le c$.
    Write $s := \Skel(x, p_t)$. 
    We should show that every $p_t$-filled block of $s$
    has either length $\ge p_{t-1} + 2c$ or length $\le p_{t-1}$.
    Suppose that the filled $p_t-$block is $s[l, h)$.
    One of the following cases must hold:
    \begin{itemize}
        \item the block contains an interval $[kp_{t-1}, (k+1)p_{t-1})$ for some $k$.
        Then it also contains 
        $[kp_{t-1} - c, kp_{t-1})$ and $[(k+1)p_{t-1}, (k+1)p_{t-1} + c)$ 
        because these positions were assumed to be $p_{t-1}-$periodic.
        Consequently it has length $\ge p_{t-1} + 2c$.
        
        \item the block is fully contained in 
        an interval $[kp_{t-1}, (k+1)p_{t-1})$ for some $k$. 
        Then trivially it has length at most $p_{t-1}$.

        \item if neither of previous cases hold, then $[l, r)$ 
        intersects two $p_{t-1}$ intervals, that is
        $$
            kp_{t-1} < l < (k+1)p_{t-1} < h < (k+2)p_{t-1}.
        $$
        So $s(l - 1) = \blank$, and $l-1 \in [kp_{t-1}, (k+1)p_{t-1})$,
        which implies, by definition of generalized Oxtoby sequence,
        that all $p_{t-1}$-holes in the interval $[kp_{t-1}, (k+1)p_{t-1})$
        are also $p_t$-holes. Similarly, $s(h) = \blank$ and
        $h \in [(k+1)p_{t-1}, (k+2)p_{t-1})$, 
        so all $p_{t-1}$-holes in the interval $[(k+1)p_{t-1}, (k+2)p_{t-1})$
        are also $p_t$-holes.
        Take any $p_{t-1}$ hole $j$ that 
        belongs to the interval $[kp_{t-1}, (k+1)p_{t-1})$.
        Then $j + p_{t-1}$ is a $p_{t-1}$-hole in $[(k+1)p_{t-1}, (k+2)p_{t-1})$.
        By the previous statement, both $j$ and $j + p_{t-1}$ are 
        also $p_t$-holes, meaning that $s(j) = \blank = s(j + p_{t - 1})$.
        This implies
        $$
            j < l < h < j + p_{t - 1},
        $$
        which implies $h - l \le p_{t - 1}$. \qedhere
    \end{itemize}
\end{proof}

\begin{lem} \cite[Proposition 12]{Kaya_2017} \label{lem:holes-mapped-close-to-holes}
    Let $X, Y$ be Toeplitz subshifts and $F \colon X \to Y$ a topological conjugacy.
    Let $r := |F|$, where $|F|$ is the larger of the widths of the block codes 
    $F, F^{-1}$. 
    Let $x \in X$ and let $y := F(x)$. 
    If $x(i)$ is a $p-$hole, then there exists $j$ such that $|i-j| < r$ and 
    $y(j)$ is a $p-$hole.
\end{lem}

\begin{proof}
    Suppose not. Then $y[i-r, i+r]$ is in $p-$skeleton of $y$. 
    Due to the block code inducing $f^{-1}$, $x(i)$ is a function of $y[i-r, i+r]$.
    So $x(i)$ is also $p-$periodic, so it is not a $p-$hole. Contradiction.
\end{proof}

Let $X$ be a generalized Oxtoby subshift w.r.t. $(p_t)_{t \ge 1}$. We define
$$
\chi(X)_{p_t} := \left\{ 
    \sigma^{\lfloor \len(W)/2 \rfloor} W \colon
    W \in \Parts_*(X, p_t) \textrm{ and } 
    len(W) > p_{t-1}
    \right\}.
$$
In informal terms, 
$\chi(X)_{p_t}$ contains the elements of $\Parts(X, p_t)$ 
who have a filled $p_t$-block of length $> p_{t-1}$ centered at zero
(or to be precise, the middle of the block should be at zero if the length of 
the block is odd, and at $-1/2$ if the length of the block is even). 

\begin{lem} \cite[Lemma 9]{Kaya_2017} \label{lem:53}
    Let $X, Y$ be Toeplitz subshifts over the alphabet $\A$, 
    $x \in X$ and $y \in Y$ any points, 
    and $\pi \colon X \to Y$ an isomorphism of subshifts such that $\pi(x) = y$. 
    Suppose further that $p > 1$, $r$ is the larger of the widths of the block codes 
    $\pi, \pi^{-1}$ and that $[-r, r] \subseteq \Per_p(x), \Per_p(y)$. 
    Then there exists $\phi \in Sym(\A^p)$
    such that 
    $ y = \tilde{\phi}(x) $
    for all
    $k \in \Z$.
\end{lem}

\begin{cor} \label{cor:54}
    Let $X, Y$ be Toeplitz subshifts over the alphabet $\A$, 
    and $\pi \colon X \to Y$ an isomorphism of subshifts. 
    Let $r$ be the larger of the widths of the block codes 
    $\pi, \pi^{-1}$.
    Suppose further that $p > 1$, $W \in \Parts(X, p)$, $V \in \Parts(Y, p)$,
    $\pi(W) = V$, 
    and the intervals $\Skel(W, p)[-r, r]$, $\Skel(V, p)[-r, r]$ do not contain
    a $\blank$.
    Then there exists $\phi \in Sym(\A^p)$
    such that $\tilde{\phi}|_W = \pi|_W$.
\end{cor}

\begin{proof}
    Take any point $x \in W$, and let $y := \pi(x) \in \pi(W) = V$.
    We have $\Skel(x, p) = \Skel(W, p)$ and $\Skel(y, p) = \Skel(V, p)$.
    Applying lemma \ref{lem:53}, we get that 
    $y = \tilde{\phi}(x)$ for some $\phi$. 
    The map $\tilde{\phi}$ commutes with $\sigma^p$, therefore
    $\tilde{\phi}(\sigma^{pk} x) = \sigma^{pk} y$ for all $k \in \Z$.
    At the same time, we have
    $\pi (\sigma^{pk} x) = \sigma^{pk} y$ for all $k \in \Z$.
    Both $\tilde{\phi}$ and $\pi$ are continuous, and coincide on the set 
    $A(x, p, 0)$. We conclude that they must be equal on the larger domain
    $\overline{A(x, p, 0)}$ which is equal to $W$.
\end{proof}

\begin{lem} \label{lem:main}
    Let $X, Y$ be two generalized Oxtoby subshifts w.r.t. $(p_t)_{t \ge 1}$.
    The following conditions are equivalent:
    \begin{enumerate}
        \item[(i)] $X, Y$ are isomorphic,
        \item[(ii)] there exists $t_0$ such that for all $t \ge t_0$ we have
        $\chi(X)_{p_t} \; D_{p_t}^{fin} \chi(Y)_{p_t}$.
    \end{enumerate}
\end{lem}

\begin{proof}
    (ii $\Rightarrow$ i) 
    For large enough $t$, we have
    $\chi(X)_{p_t} \; D_{p_t}^{fin} \chi(Y)_{p_t}$ and both
    $\chi(X)_{p_t}$ and $\chi(Y)_{p_t}$ are nonempty.
    We find 
    $A \in \chi(X)_{p_t}, B \in \chi(Y)_{p_t}$ such that $A D_{p_t} B$.
    By Lemma \ref{lem:parts_Dp}, we conclude that $X, Y$ are isomorphic.

    \noindent (i $\Rightarrow$ ii)
    Let $\pi \colon X \to Y$ be the isomorphism.
    Let $r$ be the larger of the widths 
    of the block codes inducing $\pi, \pi^{-1}$. 
    By Lemma \ref{lem:small_and_large_gap}, there exists $t_0$ such that
    for all $t \ge t_0$ we have
    \begin{itemize}
        \item $p_{t - 1} > 10r$, and
        \item every $W \in \Parts_*(X, p_t)$ has either
            $len(W) \ge p_{t-1} + 4r$, or $len(W) \le p_{t-1}$,
        \item every $V \in \Parts_*(Y, p_t)$ has either
            $len(V) \ge p_{t-1} + 4r$, or $len(V) \le p_{t-1}$.
    \end{itemize}
    We will show that $\chi(X)_{p_t} \; D_{p_t}^{fin} \chi(Y)_{p_t}$.
    Take $W \in \chi(X)_{p_t}$. 
    Pick $w \in W$.
    We have $W = \sigma^{\lfloor n / 2 \rfloor} Z$ 
    for some $Z \in \Parts_*(X, p_t)$, where $n = \len(Z)$.
    By definition of $Z$,  $w$ has $p_t-$holes at 
    $l := -1 -\lfloor n / 2 \rfloor$ and $h := n - \lfloor n / 2 \rfloor$, and
    nowhere between those two indices. 
    Next, let $u := \pi(w)$.
    By Lemma \ref{lem:holes-mapped-close-to-holes}, 
    we can find $l' < 0 < h'$ such that
    $|l - l'| < r$, $|h - h'| < r$, $u$ has $p_t-$holes at $l', h'$ 
    and nowhere between those two indices.
    By definition of $\chi(X)_{p_t}$ and second bullet point, we know that 
    $$ n = \len(Z) \ge p_{t-1} + 4r, $$
    so
    \begin{equation} \label{eqn:2}
        h' - l' + 1 \ge n - 2r + 1 > p_{t-1}.
    \end{equation}

    Let $u := \pi(w)$.
    Since $w \in W$ and $W \in \Parts(X, p_t)$, we have 
    $W = \overline{A(w, p, 0)}$ and consequently
    \begin{equation} \label{eqn:3}
        \pi (W) = 
        \pi \left( \overline{ \{\sigma^{pi} w \colon i \in \Z \} } \right) = 
        \overline{ \pi \left( \{\sigma^{pi} w \colon i \in \Z \}  \right) } = 
        \overline{ \{\sigma^{pi} u \colon i \in \Z \} } = 
        \overline{A(u, p, 0)}.
    \end{equation}
    Let $j := \lceil (l' + h') / 2 \rceil$.
    We claim that $\sigma^j ( \pi (W) ) \in \chi(Y)_{p_t}$.
    By (\ref{eqn:3}) we know that $V := \sigma^j(\pi(W)) = \overline{A(u, p, j)}$
    so $\Skel(V, p_t)$ is well-defined and equal to $\Skel(\sigma^j u, p_t)$.
    We established in the previous paragraph that 
    $\Skel(u, p_t)(l') = \blank = \Skel(u, p_t)(h') $
    and $\Skel(u, p_t)(i) \neq \blank$ for $l' < i < h'$. 
    We have chosen the value of $j$ so that 
    $\sigma^j u$ has a filled $p_t$-block centered at zero. 
    This $p_t$-block has length $h' - l' + 1 > p_{t-1}$ by (\ref{eqn:2}), 
    so indeed we have $V \in \chi(Y)_{p_t}$.
    Recall that $\pi$ is given by a block code of width $\le r$, so 
    $\sigma^j \pi$ is given by a block code of width $R := \le r + |j|$. 
    We have 
    $$
        |l' + h'| \le |l' - l| + |l - h| + |h - h'| \le r + 1 + r,
    $$
    and consequently, using first bullet point,
    $$
        R \le r + |j| \le 2r + 1 \le \lfloor p_{t-1} / 2 \rfloor.
    $$
    Since $W \in \chi(X)_{p_t}$ and 
    $V \in \chi(Y)_{p_t}$
    we have that 
    $$ \Per_{p_t}(W), \Per_{p_t} \supseteq 
    [-\lfloor p_{t-1} / 2 \rfloor, \lfloor p_{t-1} / 2 \rfloor]
    \supseteq [-R, R].
    $$
    So Corollary \ref{cor:54} applies (to the isomorphism $\sigma^j \pi$)
    and we get 
    $\sigma^j \pi|_W = \tilde{\phi}|_W$
    for a suitable bijection $\phi \colon \A^{p_t} \to \A^{p_t}$.
    In particular, this means that $V = \sigma^j (\pi(W)) = \tilde{\phi}(W)$ 
    so $W D_{p_t} V$. 
    The same reasoning applies to every $W \in \chi(X)_{p_t}$, which 
    shows that 
    $\{ [W]_{p_t} \colon W \in \chi(X)_{p_t}\} 
        \subseteq 
        \{ [V]_{p_t} \colon V \in \chi(Y)_{p_t}\} 
    $. By symmetrical reasoning 
    (swapping $X$ and $Y$)
    we get the reverse inclusion, and conclude that
    \begin{equation*}
    \chi(X)_{p_t} \; D_{p_t}^{fin} \chi(Y)_{p_t}. \qedhere
    \end{equation*}
    \end{proof}

\begin{thm}
    Let $\mathcal{G}_p$ be the family of generalized Oxtoby subshifts
    w.r.t. $p=(p_t)_{t \ge 1}$ over some alphabet $\A$.
    Then the isomorphism relation $\cong$ on $\mathcal{G}_p$ is hyperfinite.
\end{thm}

\begin{proof}
    Recall that $D_{p_t}$ is a finite equivalence relation on $K(\A^\Z)$.
    Let $C := K(\A^\Z)^{< \omega}$. 
    Then $D^{fin}_{p_t}$ is an equivalence relation on $C$. 
    Because $D_{p_t}$ is finite, $D^{fin}_{p_t}$ is also finite, so it is smooth.
    Let $F_i$ be an equivalence relation on $C^\omega$, defined by
    $$
        \mathscr{A} F_i \mathscr{B}  
        \; \Leftrightarrow \;
        \mathscr{A}_t D^{fin}_{p_t} \mathscr{B}_t \textrm{ for all } t \ge i.
    $$
    Let us define $F := \bigcup_{i \ge 1} F_i$. 
    Note that $F_i$ is smooth and $F_1 \subseteq F_2 \subseteq \dots$, 
    so $F$ is hypersmooth.
    By Lemma \ref{lem:main}, the map 
    $$
        \mathcal{G}_p \ni X \mapsto 
        \left( 
            \chi(X)_{p_t}
        \right)_{t \ge 1} \in C^\omega
    $$
    is a reduction from $\cong$ to $F$.
    So $\cong$ is hypersmooth, and also countable 
    (since each isomorphism is given by a block code, 
    and there are countably many of these).
    This implies that $\cong$ is hyperfinite 
    \cite[Theorem 8.1.5]{GaoBook}.
\end{proof}

\section{The Oxtoby property for subshifts} \label{7}

Recall that $(p_t)$ is a fixed sequence of natural numbers such that $p_t\mid p_{t+1}$ which will be used as a period structure of a generalized Oxtoby system.
\begin{defn}
     Let $X$ be a generalized Oxtoby system with respect to $(p_t)$ and $x\in X$.
      We say an interval $[a,a+p_t)$ is a \textbf{piece} in $x$ if for all $k>t$, the following hold:
     if $[a,a+p_t)$ contains one $p_{k}$-hole of $x$, then all $p_t$-holes of $x$ in $[a,a+p_t)$ are $p_k$-holes of $x$.
\end{defn}
Note that if $x$ is a Toeplitz sequence, then an interval $[a,a+p_t)$ is a piece in $x$ if and only if all $p_t$-holes of $x$ in $[a,a+p_t)$ have the same essential period.

\begin{defn}
    We say a sequence $x$ in a generalized Oxtoby subshift $(X,(p_t))$ has the \textbf{Oxtoby property} if there exists $T\in \mathbb{N}$ such that for all $t\geq T$ and $k\in \mathbb{Z}$, the interval $[kp_t, (k+1)p_t)$ is a piece in $x$. 
\end{defn}

Note that $x$ in the above definition might not be Toeplitz. Note also that every generalized Oxtoby sequence $z$ with respect to $(p_t)$ have the Oxtoby property since for every $t$ and $k$, the interval of the form $[kp_t,(k+1)p_t)$ is a piece in $z$.
\begin{lem} \label{Oxtoby cutting}
Let $(X,(p_t))$ be a generalized Oxtoby system, $z$ be a sequence in $X$ and $t\in \mathbb{N}$ such that for all $k\in \mathbb{Z}$, the interval $[kp_t,(k+1)p_t)$ is a piece of $z$. For every $x\in X$, let $a$ be a natural number such that $0\leq a<p_t$ and $x\in \overline{A(z,p_t,a)}$. Suppose $p_t$ is an essential period of $z$, then for every  $k\in \mathbb{Z}$, then the interval $[-a+kp_t,-a+(k+1)p_t)$ is a piece in $x$.    
\end{lem}
\begin{proof}
    Suppose there exists a $k\in \mathbb{Z}$ such that $[-a+kp_t,-a+(k+1)p_t)$ is not a piece in $x$. Then there exist $n_1$, $n_2$ in $[-a+kp_t,-a+(k+1)p_t)$ and $t_1>t$ such that $n_1\in {\rm Per}_{p_{t_1}}(x)\setminus {\rm Per}_{p_{t_1-1}}(x)$ and $n_2\not \in {\rm Per}_{t_1}(x)$. 
    
    By Lemma \ref{lem:parts_of_toeplitz}, $\{\overline{A(z,p_{t_1},m)}\}_{0\leq m<p_{t_1}}$, is a partition of $X$, thus, we can find $0\leq a_1<p_{t_1}$ such that $x\in \overline{A(z,p_{t_1},a_1)}$. By the definition of $\overline{A(z,p_{t_1},a_1)}$, we know $x$ has the same $p_{t_1}$-skeleton as  $S^{a_1}x$. Now we will show that $a_1\equiv a$ (mod $p_t$). Since $p_t\mid p_{t_1}$, we have ${\rm Per}_{p_t}(x)\subset  {\rm Per}_{p_{t_1}}(x)$, thus $x$ has the same $p_{t}$-skeleton as $S^{a_1}x$, in other words $x\in \overline{A(z,p_t,a_1)}$. Since $p_t$ is an essential period of $z$, by Lemma \ref{essential period}, we have $\overline{A(z,p_t,m)}=\overline{A(z,p_t,m')}$ if and only if $m\equiv m'$ (mod $p_t$). This implies that $a_1\equiv a$ (mod $p_t$). In other words, we can write $a_1$ as $a+mp_t$ for some $m\in \mathbb{Z}$.

    Now since $x$ has the same $p_{t_1}$-skeleton as $S^{a_1}z$, we know that  
    $$
    {\rm Per}_{p_{t_1}}(x)\cap [-a+kp_t,-a+(k+1)p_t)=  {\rm Per}_{p_{t_1}}(z)\cap (a_1+[-a+kp_t,-a+(k+1)p_t))
    $$
    and 
    $$
    a_1+[-a+kp_t,-a+(k+1)p_t)=a+mp_t+[-a+kp_t,-a+(k+1)p_t)
    $$
    $$
    =[(m+k)p_t,(m+k+1)p_t).
    $$
    Thus, there will be $n_1'=n_1+a_1$ and $n_2'=n_2+a_1$ both are in $[(m+k)p_t,(m+k+1)p_t)$ such that $n_1'\in {\rm Per}_{p_{t_1}}(z)\setminus {\rm Per}_{p_{t_1-1}}(z)$ but $n_2'\not \in {\rm Per}_{p_{t_1}}(z)$. This implies that $[(k+m)p_t,(k+m+1)p_t)$ is not a piece of $z$ which contradicts the choice of $z$. 
\end{proof}

\begin{cor}\label{OC1}
    Let $(X,(p_t))$ be a generalized Oxtoby system. For every Toeplitz sequence $x\in X$ and $t\in \mathbb{N}$, there exists an $0\leq a< p_t$ such that for all $k\in \mathbb{Z}$, the interval $[-a+kp_t,-a+(k+1)p_t)$ is a piece of $x$.
\end{cor}
\begin{proof}
    Take $z$ to be a generalized Oxtoby sequence in $X$. By Lemma \ref{lem:parts_of_toeplitz}, the set $\{\overline{A(z,p_t,a)}\}_{0\leq a <p_t}$ is a partition of $X$, thus we can find $0\leq a <p_t$ such that $x\in \overline{A(z,p_t,a)}$ and then apply Lemma \ref{Oxtoby cutting}.
\end{proof}

\begin{thm}\label{op}
    Suppose $(X_1,(p_t))$ and $(X_2,(p_t))$ are two generalized Oxtoby subshifts, $f:X_1 \rightarrow X_2$ is an isomorphism from $X_1$ to $X_2$. If $z\in X_1$ is Toeplitz and has the Oxtoby property, then $f(z)$ has the Oxtoby property.
\end{thm}
    
\begin{proof}
    Let $(p_t)$ be the period structure of $X_1$ and $X_2$. Since $z$ has the Oxtoby property, there exists $T\in \mathbb{N}$ such that for all $t\geq T$ and $k\in \mathbb{Z}$, the interval $[kp_t,(k+1)p_t)$ is a piece. By replacing the period structure with $(p_t)_{t\geq T}$, we may assume that $z\in X_1$ is a generalized Oxtoby sequence. 
    
    By Theorem \ref{thm:DKL}, for some $t_0\in \mathbb{N}$, there exists a  $\phi\in {\rm Sym}(\A^{p_{t_0}})$ such that for all $k\in \mathbb{Z}$,
    $$
    \phi(z[kp_{t_0},(k+1)p_{t_0}))= f(z)[kp_{t_0},(k+1)p_{t_0}).
    $$

    Since $f(z)$ is in a generalized Oxtoby subshift, by Corollary \ref{OC1}, there exists $0\leq a<p_{t_0}$ such that for all $k\in \mathbb{Z}$, we have that all $p_{t_0}$-holes of $f(z)$ in $[-a+kp_{t_0},-a+(k+1)p_{t_0})$ have the same essential period. Let $p_T$ be the essential period of all $p_{t_0}$-holes of $f(z)$ in $[-a,-a+p_{t_0})$. {\color{red}}Fix $t\geq T$ and $k\in \mathbb{Z}$,  we will show that the interval $[kp_t,(k+1)p_t)$ is a piece in $f(z)$.

     Let $\alpha=\frac{p_t}{p_{t_0}}$, we have
    \begin{equation}\label{Pt0 intervals}
        [kp_t,(k+1)p_t)= \bigsqcup_{m=0}^{\alpha-1} [kp_t+mp_{t_0},kp_t+(m+1)p_{t_0}).
    \end{equation}

    
    Suppose towards a contradiction that $[kp_t,(k+1)p_t)$ is not a piece. Then there will be $p_{t_1}< p_{t_2} $ (with $t_2>t_1>t)$ such that there are $n_1,n_2 \in [kp_t,(k+1)p_t)$ such that $n_1$ and $n_2$ have essential periods $p_{t_1}$ and $p_{t_2}$ in $f(z)$ respectively. Without the loss of generality, we may assume that $p_{t_2}$ and $p_{t_1}$ are respectively the largest  and  the second largest essential period of elements in $[kp_t,(k+1)p_t)$. We can find $k_1$ and $k_2$ such that
    $$
    -a+kp_t+k_1p_{t_0}\leq n_1<-a+kp_t+(k_1+1)p_{t_0}
    $$
    and 
    $$
    -a+kp_t+k_2p_{t_0}\leq n_2<-a+kp_t+(k_2+1)p_{t_0}.
    $$
    Write
    $$
    I_1=[-a+kp_t+k_1p_{t_0},-a+kp_t+(k_1+1)p_{t_0})
    $$ 
    and 
    $$
    I_2=[-a+kp_t+k_2p_{t_0},-a+kp_t+(k_2+1)p_{t_0}).
    $$
    By the choice of $a$, it follows that all $p_{t_0}$-holes of $f(z)$ in $I_1$ have the same essential period, we know that $m_1\neq m_2$.
    \begin{claim}\label{not left most}
         The intervals $I_1$ and $I_2$ are not the leftmost or the rightmost interval of the form $[-a+lp_{t_0},-a+(l+1)p_{t_0})$ intersecting with $[kp_t,(k+1)p_t)$. 
    \end{claim}
    \begin{proof}
        By the definition of $p_T$, the interval $[-a,-a+p_{t_0})$ has essential period $p_T$ in $f(z)$.  Since $t\geq T$, both the intervals 
    $$
    [-a+kp_t,-a+p_{t_0}+kp_t)\,\, \mbox{and}\,\, [-a+(k+1)p_t,-a+p_{t_0}+(k+1)p_t)
    $$ are of the form $[-a+mp_T,-a+p_{t_0}+mp_T)$ where $m\in \mathbb{Z}$. Therefore, both intervals  
    $$
    [-a+kp_t,-a+p_{t_0}+kp_t)\,\,\mbox{and}\,\,[-a+(k+1)p_t,-a+p_{t_0}+(k+1)p_t)
    $$ have essential period $p_T$ in $f(z)$. Since the intervals $I_1$ and $I_2$ have essential period greater than $p_T$ in $f(z)$, they are not the intervals  
    $$
    [-a+kp_t,-a+p_{t_0}+kp_t) \,\,\mbox{or}\,\,
     [-a+(k+1)p_t,-a+p_{t_0}+(k+1)p_t)
    $$ 
    which are respectively the leftmost and rightmost interval of the form $[-a+lp_{t_0},-a+(l+1)p_{t_0})$ intersecting with $[kp_t,(k+1)p_t)$.
    \end{proof} 
    
    By (\ref{Pt0 intervals}), we can find $m_1$ such that
    $$
    n_1\in [kp_t+m_1p_{t_0},kp_t+(m_1+1)p_{t_0}).
    $$
    Now let 
    $$
    J=[kp_t+m_1p_{t_0},kp_t+(m_1+1)p_{t_0}).
    $$
    \textbf{Case 1}: $J$ is disjoint with ${\rm Per}_{p_{t_2}}(f(z))$. 
    Then, we can find an integer $m_2\neq m_1$ such that
    $$
    n_2\in [kp_t+m_2p_{t_0},kp_t+(m_2+1)p_{t_0})
    $$
     Now let
    $$
     M=[kp_t+m_2p_{t_0},kp_t+(m_2+1)p_{t_0}).
    $$
    Since $p_{t_2}$ is the largest essential period of elements in 
    $$
    [kp_t,(k+1)p_t),
    $$
    we know the essential period of $M_2$ 
   in $f(z)$ is $p_{t_2}$.
    
    
    Note that because by our case assumption, the interval $M$
    is disjoint with ${\rm Per}_{p_{t_2}}(f(z))$. Thus, the essential period of 
    $J$ in $f(z)$ could only be $p_{t_1}$ or lower. But $n_1\in J$ and the essential period of $n_1$ in $f(z)$ is $p_{t_1}$. Thus, the essential period of 
    $J$ in $f(z)$ is $p_{t_1}$.
    Since $\phi \in {\rm Sym}(\A^{p_{t_0}})$ is a bijection, we know that for any $l\in \mathbb{Z}$, the interval $[lp_{t_0},(l+1)p_{t_0})$ must have the same essential period in $z$ as it has in $f(z)$.  This implies
    that the essential period of 
    $J$ in $z$ is $p_{t_1}$ and the essential period of 
    $M$ in $z$ is $p_{t_2}$. Since both intervals 
    $J$ and $M$ 
    are subsets of $kp_t,(k+1)p_t)$, this implies that the interval $[kp_t,(k+1)p_t)$ contains $p_t$-holes of $z$ which have essential period $p_{t_1}$ and $p_{t_2}$. In other words, $[kp_t,(k+1)p_t)$ is not a piece in $z$. This contradicts the fact that $z$ is a generalized Oxtoby sequence w.r.t. to $(p_t)$.

\begin{figure}[h]
    \centering
    \begin{tikzpicture}

        \draw[thick] (-2,0) rectangle (7,1);
        
        \node[below] at (-2,0) {\( kp_t \)};
        \node[below] at (7,0) {\( (k+1)p_t\)};

        \def\subwidth{2} 

        \filldraw[blue!50] (-0.5,0) rectangle (-0.5+\subwidth,1);
        \node at (-0.5+\subwidth/2,0.5) {\( J \)};

        \filldraw[blue!50] (4,0) rectangle (4+\subwidth,1);
        \node at (4+\subwidth/2,0.5) {\( M \)};

       \end{tikzpicture}
    \caption{Case 1}
    \label{fig:adjacent_intervals}
\end{figure}

     \textbf{Case 2}: The interval $J$
    intersects ${\rm Per}_{p_{t_2}}(f(z))$. Possibly replacing $n_2$ with another element in ${\rm Per}_{p_{t_2}}(f(z))$, we may assume that $n_2\in J$.
    Since the roles of $n_1$ and $n_2$ are symmetric we may assume that $n_1<n_2$.

    Note that $k_1=k_2-1$ otherwise the distance between the right endpoint of $I_1$ and the left endpoint of $I_2$ is a multiple of $p_{t_0}$, but $n_1,n_2\in J$ which has length $p_{t_0}$. Note that $n_1\in J\cap I_1$, thus, the left endpoint of $I_1$ is less than the right endpoint of $J$ and the right endpoint of $I_1$ is greater than the left endpoint of $J$ which implies that
    \begin{equation}\label{E1}
        -a+kp_t+k_1p_{t_0}<kp_t+(m_1+1)p_{t_0}
    \end{equation}
      and
    \begin{equation} \label{E2}
         kp_t+m_1p_{t_0}<-a+kp_t+(k_1+1)p_{t_0}.
    \end{equation}
   
    Since $a<p_{t_0}$, (\ref{E1}) implies that 
    $$
    k_1\leq m_1+1
    $$
    and (\ref{E2}) implies that
    $$
    m_1\leq k_1.
    $$
    Thus, we have
    $$
    m_1 \leq k_1\leq m_1+1.
    $$
    Also, $n_2\in J\cap I_2$, analogous argument shows that
    $$
    m_1 \leq k_2=k_1+1 \leq m_1+1.
    $$
    Thus, $m_1=k_1$. So we have
    $$
    I_1=[-a+kp_t+m_1p_{t_0},-a+kp_t+(m_1+1)p_{t_0})
    $$ and
    $$
    I_2=[-a+kp_t+(m_1+1)p_{t_0},-a+kp_t+(m_1+2)p_{t_0})
    $$
    Let
    $$
    J_1= [kp_t+(m_1-1)p_{t_0},kp_t+m_1p_{t_0})
    $$
    \begin{claim}\label{subset}
        $J_1,J \subset [kp_t,(k+1)p_t)$.
    \end{claim}
    \begin{proof}
     Note that the leftmost interval of the form $[-a+lp_{t_0},-a+(l+1)p_{t_0})$ intersecting $[kp_t,(k+1)p_t)$ is 
    $$
    [-a+kp_t,-a+kp_t+p_{t_0}].
    $$
    By Claim \ref{not left most}, the interval $I_1$ is not the leftmost interval of the form $[-a+lp_{t_0},-a+(l+1)p_{t_0})$ intersecting $[kp_t,(k+1)p_t)$ and $I_1\subset [kp_t,(k+1)p_t)$, we have the left endpoint of $I_1$ is at least $-a+kp_t+p_{t_0}$. In other words,
    $$
    -a+kp_t+m_1p_{t_0}>-a+kp_t+p_{t_0}
    $$
    thus, $m_1\geq 1$. The rightmost interval of the form $[-a+lp_{t_0},-a+(l+1)p_{t_0})$ intersecting $[kp_t,(k+1)p_t)$ is
    $$
    [-a+(k+1)p_t,-a+(k+1)p_t+p_{t_0}).
    $$
     Now the left endpoint of $J_1$ is
     $$
      kp_t+(m_1-1)p_{t_0}\geq kp_t.
     $$
     The claim follows from the fact that $J_1<J$.

\begin{figure}[h]
    \centering

        \begin{tikzpicture}

        \draw[thick] (-2,0) rectangle (7,1);
        
        \node[below] at (-2,0) {\( kp_t \)};
        \node[below] at (7,0) {\( (k+1)p_t\)};

        \def\subwidth{2} 

        \filldraw[blue!50] (-0.5,0) rectangle (-0.5+\subwidth,1);
        \node at (-0.5+\subwidth/2,0.5) {\( J_1\)};

        \filldraw[blue!50] (-0.5+\subwidth,0) rectangle (-0.5+2*\subwidth,1);
        \node at (-0.5+1.5*\subwidth,0.5) {\( J \)};

        \draw[thick] (-0.5+\subwidth,0) -- (-0.5+\subwidth,1);

       \end{tikzpicture}

    \caption{Case 2}
    \label{fig:adjacent_intervals}
\end{figure}

    \end{proof}

    Recall that $I_1$ and $I_2$ are pieces of $f(z)$. Note that $n_1\in I_1$ and $n_1$ has essential period $p_{t_1}$. This implies that all $p_{t_0}$ holes in $I_1$ have essential period $p_{t_1}$. Note that $n_2\in I_2$ and $n_2$ has essential period $p_{t_2}$. This implies all $p_{t_0}$-holes in $I_2$ have essential period $p_{t_2}$. Note that $\phi$ preserves the essential period of an interval of the form $[lp_{t_0},(l+1)p_{t_0})$ in $z$ and $f(z)$. Since $p_{t_2}$ is the largest essential period of elements in $[kp_t,(k+1)p_t)$ in $f(z)$, we know that there will be an interval of the form $[lp_{t_0},(l+1)p_{t_0})$ with essential period $p_{t_2}$ in $z$. Since $[kp_t,(k+1)p_t)$ is a piece in $z$, all intervals of the form $[lp_{t_0},(l+1)p_{t_0})$ in $[kp_t,(k+1)p_t)$ must have essential period $p_{t_2}$ in $f(z)$. Now let 
    $$
    I_0=[-a+(m_1-1)p_{t_0},-a+m_1p_{t_0})=I_1-p_{t_0}.
    $$
    Note that $J_1\subset I_0\cup I_1$. Since $n_1$ is a $p_{t_0}$-hole of $f(z)$ in $I_1$, we have $n_1-p_{t_0}$ is a $p_{t_0}$-hole of $f(z)$ in $I_0$. The interval $J_1$ must have essential period $p_{t_2}$ in $f(z)$, but all $p_{t_0}$-holes of $f(z)$ in $I_1\cap J_1$ have essential period $p_{t_1}$ in $f(z)$. Hence there must be at least one $p_{t_0}$-hole of $f(z)$ in $I_0 \cap J_1$ which has essential period $p_{t_2}$ in $f(z)$.  Since $I_0$ is a piece of $f(z)$, all $p_{t_0}$-holes of $f(z)$ in $I_0$ have essential period $p_{t_2}$. In the end, we will have three consecutive pieces 
    $$
    I_0<I_1<I_2
    $$ of length $p_{t_0}$ such that $I_0$ and $I_2$ has essential period $p_{t_2}$ but $I_1$ has essential period $p_{t_1}$. Note that $p_{t_2}>p_{t_1}>p_t$, it is impossible to find an $0\leq b<p_t$ with $b\equiv a$ (mod $p_{t_0}$) such that for all $k\in \mathbb{Z}$, the interval $[-b+kp_t,-b+(k+1)k_{p_t})$ have the Oxtoby property which contradicts Corollary \ref{OC1}.
    \end{proof}


    
\section{Second proof of hyperfiniteness of conjugacy of generalized Oxtoby subshifts} \label{8}
\begin{defn} \label{ER}
    Let $T\in \mathbb{N}$. We define the relation $F_T$ on the space of generalized Oxtoby systems with respect to $(p_t)$ as follows: $XF_TY$ if there exists $W\in {\ Parts}(X,p_T)$ and $V\in {\ Parts}(Y,p_T)$ such that
    \begin{enumerate}
        \item $\forall x\in W \forall y\in V$ and $k\in \mathbb{Z}$, the interval $[kp_T,(k+1)p_T)$ is a piece of $x$ and $y$.
        \item  There exists $\tilde \phi \in {\rm Sym}(\A^{p_T})$ such that $\phi(W)= V$.
    \end{enumerate}
\end{defn}
\begin{lem}\label{ANY DEF}
         Let $X$ and $Y$ be generalized Oxtoby systems w.r.p. to $(p_t)$. Then $XF_TY$ if and only if there exists $x\in X$, $y\in Y$ such that
         \begin{itemize}
             \item  the interval $[kp_T,(k+1)p_T)$ is a piece for both $x$ and $y$,
             \item  there exists $\phi \in {\rm Sym}(\A^{p_T})$ such that  we have $y=\tilde \phi (x)$.
         \end{itemize}
    \end{lem}

    \begin{proof}
        ($\Rightarrow$) This is by the definition of $F_T$, one can take any $x\in W$ and $y=\tilde \phi (x)\in V$.
        
        ($\Leftarrow$) Take $W$ to be $\overline{A(x,p_T,0)}$ and $V$ to be  $\overline{A(y,p_T,0)}$, by Lemma \ref{Oxtoby cutting} applied with $a=0$, we know for any $x'\in \overline{A(x,p_T,0)}$ and any $y'\in \overline{A(y,p_T,0)}$ and all $k\in \mathbb{Z}$, the interval $[kp_T,(k+1)p_T)$ is a piece of $x'$ and $y'$. Thus, we have $XF_TY$. 
    \end{proof}
\begin{lem}\label{Borel EQ}
    For any $T\in \mathbb{N}$, the relation $F_T$ is a Borel set.
\end{lem}
\begin{proof}
    The definition is clearly co-analytic. By Lemma \ref{ANY DEF} it is also analytic.
    \end{proof}

\begin{lem}\label{EQ}
         For any $T\in \mathbb{N}$, the relation $F_T$ is a finite equivalence relation.
    \end{lem}

\begin{proof}
     Note that there are finitely many elements in both ${\rm Parts}(X,(p_T))$ and ${\rm Sym}(\A^{p_T})$, thus $F_T$ is finite.
    
    Now we prove that $F_T$ is an equivalence relation. Reflexivity and symmetry are obvious. We will check transitivity. Let $X F_T Y$ and $Y F_T Z$. By Lemma \ref{ANY DEF}, there exists $x\in X$ and $y\in Y$ such that both items in Lemma \ref{ANY DEF} are satisfied. Also, there exists $y'\in Y$ and $z\in Z$ such that both items in Lemma \ref{ANY DEF} are satisfied for $y'$ and $z$. We will show that both items in Lemma \ref{ANY DEF} are also satisfied for some $x'\in X$ and $z$, which will prove transitivity.

    Now let $0\leq a<p_T$ be the natural number such that $y'\in \overline{A(y,p_T,a)}$. By Lemma \ref{Oxtoby cutting}, we have that for all $k\in \mathbb{Z}$, both intervals $[kp_T,(k+1)p_T)$ and $[-a+kp_T,-a+(k+1)p_T)$ are pieces in $y'$. Now let 
    $$
    I_k=[kp_T,-a+(k+1)p_T]\,\,\mbox{and}\,\,
    J_k=[-a+(k+1)p_T,(k+1)p_T].
    $$
    \begin{claim}\label{Inter}
        Either $\bigcup_{k\in \mathbb{Z}} I_k\subset {\rm Per}_{p_T}(y')$ or $\bigcup_{k\in \mathbb{Z}} J_k\subset {\rm Per}_{p_T}(y')$.
    \end{claim}
    \begin{proof}
         We assume both $\bigcup_{k\in \mathbb{Z}} I_k$ and $\bigcup_{k\in \mathbb{Z}} J_k$ contain a $p_T$-hole, this will lead to a contradiction. Note that $I_{k+1}=I_k+p_T$, thus if one of the intervals $I_k$ contains a $p_T$-hole, then all of intervals $I_k$ contain a $p_T$-hole. Also, $J_{k+1}=J_k+p_T$, thus if one of the intervals $J_k$ contains a $p_T$-hole, then all of intervals $J_k$ contain a $p_T$-hole

        Since $I_k\cup J_k=[kp_T,(k+1)p_T)$ which is a piece in $y'$, we know that all $p_T$-holes in $I_k\cup J_k$ have the same essential period. Since $I_k\cup J_{k-1}= [-a+(k-1)p_T,-a+kp_T)$ which is also a piece in $y'$, all $p_T$-holes in $I_k\cup J_{k-1}$ have the same essential period. Keep doing this process in both directions and we get that all $p_T$-holes in the whole bi-infinite sequence $y'$ have the same essential period. This implies that $y'$ is periodic which a contradiction since there is no periodic point in a non-trivial Oxtoby system.
    \end{proof}
    \begin{claim}
        One of the following holds:
        \begin{itemize}
            \item for all $k\in \mathbb{Z}$, the sequence $y'$ restricted on $I_k$ is the same finite sequence.
            \item for all $k\in \mathbb{Z}$, the sequence $y'$ restricted on $J_k$ is the same finite sequence.
        \end{itemize}
    \end{claim}
    \begin{proof}
        By Claim \ref{Inter}, we know that either $\bigcup_{k\in \mathbb{Z}} I_k\subset {\rm Per}_{p_T}(y')$ or $\bigcup_{k\in \mathbb{Z}} J_k\subset {\rm Per}_{p_T}(y')$. If $\bigcup_{k\in \mathbb{Z}} I_k\subset {\rm Per}_{p_T}(y')$, since $I_k=I_{k-1}+p_T$, we know $I_k=I_{k-1}$ for all $k$, thus the sequence $y'$ restricted on $I_k$ is the same finite sequence. Similarly, when $\bigcup_{k\in \mathbb{Z}} J_k\subset {\rm Per}_{p_T}(y')$, we have for all $k\in \mathbb{Z}$, the sequence $y'$ restricted on $J_k$ is the same finite sequence.
    \end{proof}
   Without loss of generality, we may assume that for all $k\in \mathbb{Z}$, the sequence $y'$ restricted on $I_k$ is the same finite sequence.
   By the choice of $x$ and $y$, we know that there is a $\phi_1\in {\rm Sym}(\mathcal{A}^{p_T})$ such that for all $k
\in \mathbb{Z}$, we have 
   $$
   \phi_1(y[kp_T,(k+1)p_T))=x[kp_T,(k+1)p_T).
   $$

    We define $\phi\in {\rm Sym}(\mathcal{A}^{p_T})$ as follows:
    $$
    \phi(y'[kp_T,(k+1)p_T))=\phi_1(y'[-a+(k+1)p_T,-a+(k+2)p_T)).
    $$
    \begin{claim}
        $\phi$ is well-defined.
    \end{claim}
    \begin{proof}
         Suppose 
         \begin{equation} \label{CON}
             y'[kp_T,(k+1)p_T)=y'[k'p_T,(k'+1)p_t)
         \end{equation}
         we need to show that
         $$
         \phi(y'[(k+1)p_T,(k+2)p_T))=\phi(y'[(k'+1)p_T,(k'+2)p_t)).
         $$
         Note that $[kp_T,(k+1)p_T)=I_k\cup J_k$ and $[k'p_T,(k'+1)p_t)=I_{k'}\cup J_{k'}$. Since $y'$ restricted on $I_k$ is the same sequence for every $k$ and we know 
         \begin{equation} \label{CON1}
             y'[(k+1)p_T,-a+(k+2)p_T)=y'[(k'+1)p_T,-a+(k'+2)p_t)
         \end{equation}
         By (\ref{CON}), we know that 
         $$
             y'[-a+(k+1)p_T,(k+1)p_T)=y'[-a+(k'+1)p_T,(k'+1)p_t)
        $$
       which together with (\ref{CON1}) implies
         $$
         y'[-a+(k+1)p_T,-a+(k+2)p_t]=y'[-a+(k'+1)p_T,-a+(k'+2)p_t].
        $$
         Thus we have 
        $$
     \phi_1(y'[-a+(k+1)p_T,-a+(k+2)p_T))= \phi_1 (y'[-a+(k'+1)p_T,-a+(k'+2)p_t))
        $$
        which implies that
        $$
        \phi(y'[kp_T,(k+1)p_T))=\phi(y'[k'p_T,(k'+1)p_t)).
        $$
    \end{proof}
     By the choice of $y'$ and $z$, we know there is a $\phi_2\in {\rm Sym}(\mathcal{A}^{p_T})$ such that for all $k
\in \mathbb{Z}$, we have 
$$
        \phi_2(y'[kp_T,(k+1)p_T))=x[kp_T,(k+1)p_T)
 $$   
     The composition of $\phi_2^{-1}$ and $\phi$ is in ${\rm Sym}(\mathcal{A}^{p_T})$ such that $z$ and $\phi\circ \phi_2^{-1}(z)$ satisfying both conditions.

\end{proof}
\begin{lem}\label{End Game}
    Given two generalized Oxtoby subshifts $X, Y$, the following are equivalent
    \begin{enumerate}
    
        \item $X$ and $Y$ are conjugate.
        \item There is $T\in \mathbb{N}$ such that for all $t\geq T$, we have $XF_tY$.
    \end{enumerate}
\end{lem}
\begin{proof}
     (1)$\Rightarrow$(2) Let $z$ be a generalized Oxtoby sequence in $X$, let $f$ be an isomorphism from $X$ to $Y$. By Theorem \ref{op}, $f(z)$ has the Oxtoby property. In other words, there is $T_1\in \mathbb{N}$ such that for all $t\geq T_1$ and $k\in \mathbb{Z}$ the interval $[kp_t,(k+1)p_t)$ is a piece of $y$. By Lemma \ref{lem:53}, we can find $T_2\in \mathbb{N}$ such that for all $t\geq T_2$, there exists $\phi\in {\rm Sym}(\A^{p_{t}})$ such that $\tilde \phi(z)=y$. Take $T={\rm max}\{T_1,T_2\}$. 
     By Lemma \ref{ANY DEF}, we are done.

     (2)$\Leftarrow$(1) Fix any $t\geq T$. Let $W$ and $V$ and $\phi$ be in Definition \ref{ER} for $F_t$. Take any $x\in W$. We know that $\tilde\phi$ is an isomorphism sending $x\in W$ to $\tilde \phi (x)\in V$. Since $X$ and $Y$ are minimal, $\tilde\phi $ is an isomorphism from $X$ to $Y$.
\end{proof}
\begin{thm}
    The conjugacy relation of generalized Oxtoby systems with respect to $(p_t)$ is hyperfinite.
\end{thm}
\begin{proof}
    By Lemma \ref{End Game}, we have the conjugacy relation of generalized Oxtoby systems with respect to $(p_t)$ is equal to $\bigcup_{T\in \mathbb{N}} \bigcap_{t\geq T}F_t$.
\end{proof}
\subsubsection*{Acknowledgements.}
Konrad Deka was supported by National Science Centre, Poland, 
grant Preludium Bis no. 2019/35/O/ST1/02266. Bo Peng is partly funded by NSERC Discovery Grant RGPIN-2020-05445.

\bibliographystyle{plain} 
\bibliography{bibliography} 

\begin{bibdiv}
\begin{biblist}

\bib{DKL}{article}{
      author={Downarowicz, T.},
      author={Kwiatkowski, J.},
      author={Lacroix, Y.},
       title={A criterion for {T}oeplitz flows to be topologically isomorphic and applications},
    language={eng},
        date={1995},
     journal={Colloquium Mathematicae},
      volume={68},
      number={2},
       pages={219\ndash 228},
         url={http://eudml.org/doc/210306},
}

\bib{DownarowiczChoquet}{article}{
      author={Downarowicz, Tomasz},
       title={The {C}hoquet simplex of invariant measures for minimal flows},
        date={1991},
     journal={Israel J. Math},
      volume={74},
       pages={241\ndash 256},
}

\bib{DownarowiczSurvey}{article}{
      author={Downarowicz, Tomasz},
       title={Survey of odometers and {T}oeplitz flows},
        date={2005},
     journal={Algebraic and topological dynamics, Contemp. Math.},
      volume={385},
       pages={7–37},
}

\bib{GaoBook}{book}{
      author={Gao, Su},
       title={Invariant descriptive set theory.},
   publisher={CRC Press},
        date={2008},
}

\bib{GaoJacksonSeward}{article}{
      author={Gao, Su},
      author={Jackson, Steve},
      author={Seward, Brandon},
       title={Group colorings and bernoulli subflows},
        date={2015},
     journal={Memoirs of the American Mathematical Society},
      volume={241},
       pages={241pp},
}

\bib{Kaya_2017}{article}{
      author={Kaya, Burak},
       title={The complexity of the topological conjugacy problem for {T}oeplitz subshifts},
        date={2017-05},
        ISSN={1565-8511},
     journal={Israel Journal of Mathematics},
      volume={220},
      number={2},
       pages={873–897},
         url={http://dx.doi.org/10.1007/s11856-017-1537-4},
}

\bib{Oxtoby1952}{article}{
      author={Oxtoby, John},
       title={Ergodic sets},
        date={1952},
     journal={Bull. Amer. Math. Soc.},
      volume={58},
      number={2},
       pages={116\ndash 136},
}

\bib{MScompleteness}{article}{
      author={Sabok, Marcin},
       title={Completeness of the isomorphism problem for separable c*-algebras},
        date={2016},
     journal={Inventiones mathematicae},
      volume={204},
       pages={833\ndash 868},
}

\bib{ST}{article}{
      author={Sabok, Marcin},
      author={Tsankov, Todor},
       title={On the complexity of topological conjugacy of {T}oeplitz subshifts},
        date={2017},
     journal={Israel J. Math},
      volume={220},
       pages={583\ndash 603},
}

\bib{Simon-Thomaz}{article}{
      author={Thomas, Simon},
       title={Topological full groups of minimal subshifts and the classification problem for finitely generated complete groups},
        date={2019},
     journal={Groups, Geometry and Dynamics},
      volume={13},
       pages={327\ndash 347},
}

\bib{Williams}{article}{
      author={Williams, Susan},
       title={Toeplitz minimal flows which are not uniquely ergodic},
        date={1984},
     journal={Zeitschrift fur Wahrscheinlichkeitstheorie und Verwandte Gebiete},
      volume={67},
       pages={95 \ndash  107},
}

\end{biblist}
\end{bibdiv}

\end{document}